\documentclass[12pt,bezier]{article}
\usepackage{times}
\usepackage{booktabs}
\usepackage{pifont}
\usepackage{floatrow}
\usepackage{caption}
\usepackage{mathrsfs}
\usepackage[fleqn]{amsmath}
\usepackage{amsfonts,amsthm,amssymb,mathrsfs,bbding}
\usepackage{txfonts}
\usepackage{graphics,multicol}
\usepackage{graphicx}
\usepackage{color}
\usepackage{amssymb}
\usepackage{caption}
\usepackage{cite}
\usepackage{latexsym,bm}
\usepackage{indentfirst}
\usepackage{color}
\usepackage[colorlinks=true,anchorcolor=blue,filecolor=blue,linkcolor=blue,urlcolor=blue,citecolor=blue]{hyperref}
\usepackage{extarrows}
\usepackage{cite}
\usepackage{latexsym,bm}
\usepackage{mathtools}
\evensidemargin=\oddsidemargin\topmargin=-1.5cm

\newtheorem{thm}{Theorem}[section]
\newtheorem{case}{Case}
\newtheorem{subcase}{Case}[case]
\newtheorem{prob}{Problem}[section]
\newtheorem{claim}{Claim}

\newtheorem{lemma}{Lemma}[section]
\newtheorem{cor}{Corollary}[section]

\newtheorem{observation}{Observation}[section]

\theoremstyle{definition}

\addtocounter{section}{0}

\begin{document}
\title{Spectral radius and rainbow Hamiltonicity in bipartite graphs\footnote{Supported by the National Natural Science Foundation of China
{(No. 12371361)} and Distinguished Youth Foundation of Henan Province {(No. 242300421045)}.}}
\author{\bf  Meng Chen$^{a}$, {\bf Ruifang Liu$^{a}$}\thanks{Corresponding author.
E-mail addresses: cm4234635@163.com (M. Chen), rfliu@zzu.edu.cn (R. Liu), yuanqixuan6@gmail.com (Q. Yuan).}, \bf Qixuan Yuan$^{a}$\\
{\footnotesize $^a$ School of Mathematics and Statistics, Zhengzhou University, Zhengzhou, Henan 450001, China} }
\date{}

\maketitle
{\flushleft\large\bf Abstract}
Let $\mathcal{G}=\{G_1, G_2, \ldots , G_k\}$ be a family of bipartite graphs  on the same vertex set. A rainbow Hamilton path (cycle) in $\mathcal{G}$ is a path (cycle) that visits each vertex precisely once such that any two edges belong to different graphs of $\mathcal{G}.$ In this paper, by adopting the technique of bi-shifting, we present tight sufficient conditions in terms of the spectral radius for a family $\mathcal{G}$ to admit a rainbow Hamilton path and cycle, respectively. Meanwhile, we completely characterize the corresponding spectral extremal graphs.

\begin{flushleft}
\end{flushleft}

\section{Introduction}
Throughout this paper, we only consider simple and undirected graphs. Let $V(G)$ and $E(G)$ be the set of vertices and edges of $G,$ respectively. For any vertex $u \in V(G),$ we denote by $N_G(u)$ the neighborhood of $u$ in $G.$ Let $G$ be a bipartite graph with bipartitions $(X, Y).$ Denote by $\widehat{G}$ its quasi-complement, where $V(\widehat{G})=V(G)$ and for any $x\in X$ and $y\in Y,$ $xy\in E(\widehat{G})$ if and only if $xy\notin E(G).$ Let $G_1$ and $G_2$ be two bipartite graphs with the bipartition $(X_1, Y_1)$ and $(X_2, Y_2),$ respectively. We use $G_1\sqcup G_2$ to denote the bipartite graph obtained from $G_1 \cup G_2$ by adding all possible edges between $X_1$ and $Y_2$ and all possible edges between $Y_1$ and $X_2.$

The adjacency matrix of $G$ is the $|V(G)|\times|V(G)|$ matrix $A(G)=(a_{uv}),$ where $a_{uv}$ is the number of edges joining $u$ and $v.$ The eigenvalues of $G$ are defined as the eigenvalues of its adjacency matrix $A(G).$ The maximum modulus of eigenvalues of $A(G)$ is called the spectral radius of $G$ and denoted by $\rho(G).$

Let $\mathcal{G}= \{G_1, G_2,\ldots, G_k\}$ be a family of not necessarily distinct graphs with the same vertex
set $V.$ Let $H$ be a graph with $k$ edges on the vertex set $V(H)\subseteq V.$ We say that $\mathcal{G}$
contains a rainbow copy of $H$ if there exists a bijection $\phi: E(H) \rightarrow[k]$ such that $e \in E(G_{\phi(e)})$ for all $e\in E(H),$ where $[k]=\{1, 2, \ldots, k\}.$ In other words, each edge of $H$ comes from a different graph $G_i$. In particular, if
$G_1=G_2=\cdots=G_k$, then the rainbow copy of H reduces to the classical copy of $H.$ The following general question was proposed by Joos and Kim in \cite{joos2020rainbow}.
\begin{prob}\label{prob1.1}
Let $H$ be a graph with $k$ edges and $\mathcal{G}= \{G_1, G_2,\ldots, G_k\}$ be a family of not necessarily distinct graphs on the same vertex set $V.$ Which properties and constraints imposed on the family $\mathcal{G}$ can yield a rainbow copy of $H$?
\end{prob}
Focusing on Problem \ref{prob1.1}, Guo et al.\cite{guo2023spectral} provided a sufficient condition in terms of the spectral radius for the existence of a rainbow matching in a family of graphs. He, Li and Feng\cite{he2024spectral} proposed a spectral radius condition for a family of graphs to admit a rainbow Hamilton path. Moreover, they also posed a spectral radius condition on a family of graphs to guarantee a rainbow linear forest of given size. Zhang and van Dam\cite{zhang2025} gave a sufficient condition based on the size and the spectral radius for the existence of a rainbow Hamilton cycle in a family of graphs, respectively. Zhang and Zhang\cite{zhangzhang2025} proposed a sufficient condition in terms of the spectral radius for a family of graphs to contain a rainbow $k$-factor.

 The bipartite version of Problem \ref{prob1.1} is also a natural and intriguing problem. Bradshaw\cite{Bradshaw2021} proposed  minimum degree conditions for a family of bipartite graphs to admit a rainbow Hamilton cycle and a rainbow perfect matching, respectively. Furthermore, Bradshaw\cite{Bradshaw2021} proved a stronger result which states that a family of bipartite graphs is bipancyclic under the minimum degree condition. Motivated by \cite{Bradshaw2021}, Hu et al.\cite{Hu2024} investigated the minimum degree condition of vertex-bipancyclicity in a family of bipartite graphs. Shi, Li and Chen\cite{Shi2024} provided a sufficient condition in terms of the spectral radius for the existence of a rainbow matching in a family of bipartite graphs.

A bipartite graph with the bipartition $(X, Y)$ is called balanced if $|X|=|Y|.$ Li and Ning\cite{li2017spectral} provided a spectral radius condition for a balanced bipartite graph with given minimum degree to contain a Hamilton path.
Denote $Q^k_n =K_{k, n-k-1}\sqcup \widehat{K_{n-k, k+1}}$ and $R^k_n =K_{k, k}\sqcup \widehat{K_{n-k, n-k}}.$
\begin{thm}[Li and Ning\cite{li2017spectral}]\label{thm1.1}
Let $G$ be a balanced bipartite graph on $2n$ vertices with minimum degree $\delta(G) \geq k,$ where $k \geq 0$ and $n \geq (k+2)^2.$\\
(i) If $k\neq 1$ and $\rho(G)\geq \rho(Q^k_n),$ then $G$ admits a Hamilton path unless $G\cong Q^k_n.$ \\
(ii) If $k = 1$ and $\rho(G)\geq \rho(R^k_n),$ then $G$ admits a Hamilton path unless $G\cong R^k_n.$
\end{thm}
By Theorem \ref{thm1.1}, one can obtain the following result.
\begin{cor}[Li and Ning\cite{li2017spectral}]\label{cor1.1}
Let $G$ be a balanced bipartite graph on $2n$ vertices with $n \geq 4.$ If
\begin{eqnarray*}
\rho(G)\geq \rho(K_{n,n-1}\cup K_1),
\end{eqnarray*}
 then $G$ admits a Hamilton path unless $G\cong K_{n,n-1}\cup K_1.$
\end{cor}
A bipartite graph with the bipartition $(X, Y)$ is called nearly balanced if $|X|-|Y|= 1$ (by the symmetry). Li and Ning\cite{li2017spectral} also proposed a spectral radius condition for a nearly balanced bipartite graph with given minimum degree to contain a Hamilton path. Define $S^k_n =K_{k, n-k-1}\sqcup \widehat{K_{n-k, k}}$ and $T^k_n =K_{k, n-k-1}\sqcup \widehat{K_{n-k-1, k+1}}.$
\begin{thm}[Li and Ning\cite{li2017spectral}]\label{thm1.2}
Let $G$ be a nearly balanced bipartite graph on $2n-1$ vertices with minimum degree $\delta(G) \geq k,$ where $k \geq 0$ and $n \geq (k+1)^2.$\\
(i) If $k\geq 1$ and $\rho(G)\geq \rho(S^k_n),$ then $G$ admits a Hamilton path unless $G\cong S^k_n.$\\
(ii) If $k = 0$ and $\rho(G)\geq \rho(T^0_n),$ then $G$ admits a Hamilton path unless $G\cong T^0_n.$
\end{thm}
According to Theorem \ref{thm1.2}, one can deduce the following result.
\begin{cor}[Li and Ning\cite{li2017spectral}]\label{cor1.2}
Let $G$ be a nearly balanced bipartite graph on $2n-1$ vertices. If
\begin{eqnarray*}
\rho(G)\geq \rho(K_{n-1,n-1}\cup K_1),
\end{eqnarray*}
then $G$ admits a Hamilton path unless $G\cong K_{n-1,n-1}\cup K_1.$
\end{cor}
Based on Corollaries \ref{cor1.1} and \ref{cor1.2}, a natural and interesting problem arises.
\begin{prob}\label{prob1.2}
What is the sufficient condition in terms of the spectral radius for the existence of a rainbow Hamilton path in a family of bipartite graphs?
\end{prob}
Focusing on Problem \ref{prob1.2}, we provide sufficient conditions based on the spectral radius for the existence of a rainbow Hamilton path in a family of bipartite graphs.
\begin{thm}\label{thm1.4}
    Let $\mathcal{G}=\{G_1, G_2, \ldots , G_{2n-1}\}$ be a family of balanced bipartite graphs on vertex set $[2n]$ and the bipartition $(X, Y),$ where $n\geq2.$ If
    \begin{eqnarray*}
        \rho (G_i) \geq \rho(K_{n,n-1}\cup K_1)
    \end{eqnarray*}
     for every $i\in [2n-1],$ then $\mathcal{G}$ admits a rainbow Hamilton path unless $G_1=G_2=\cdots =G_{2n-1}\cong K_{n,n-1}\cup K_1.$
\end{thm}
\begin{thm}\label{thm1.5}
    Let $\mathcal{G}=\{G_1, G_2, \ldots , G_{2n-2}\}$ be a family of nearly balanced bipartite graphs on vertex set $[2n-1]$ and the bipartition $(X, Y ),$ where $n \geq 2.$ If
    \begin{eqnarray*}
        \rho (G_i) \geq \rho(K_{n-1,n-1}\cup K_1)
    \end{eqnarray*}
     for every $i\in [2n-2],$ then $\mathcal{G}$ admits a rainbow Hamilton path unless $G_1=G_2=\cdots =G_{2n-2}\cong K_{n-1,n-1}\cup K_1.$
\end{thm}
Li and Ning\cite{li2016spectral} also established a spectral radius condition for bipartite graphs with given minimum degree to guarantee the existence of rainbow Hamilton cycles. Denote $B^k_n =K_{k, n-k}\sqcup \widehat{K_{n-k, k}}.$
\begin{thm}[Li and Ning\cite{li2016spectral}]\label{thm1.3}
Let $G$ be a balanced bipartite graph on $2n$ vertices with minimum degree $\delta(G) \geq k\geq 1.$ If $n\geq(k+1)^2$ and $\rho(G)\geq \rho(B^k_n),$ then $G$ admits a Hamilton cycle unless $G\cong B^k_n.$
\end{thm}
By Theorem \ref{thm1.3}, we can immediately obtain the following corollary.
\begin{cor}[Li and Ning\cite{li2016spectral}]\label{cor1.3}
Let $G$ be a balanced bipartite graph on $2n$ vertices, where $n\geq 4$. If
\begin{eqnarray*}
    \rho(G)\geq \rho(K_{1, n-1}\sqcup \widehat{K_{n-1, 1}}),
\end{eqnarray*}
then $G$ admits a Hamilton cycle unless $G\cong K_{1, n-1}\sqcup \widehat{K_{n-1, 1}}.$
\end{cor}

Motivated by Corollary \ref{cor1.3}, we propose the following problem.

\begin{prob}\label{prob1.3}
What is the sufficient condition in terms of the spectral radius for the existence of a rainbow Hamilton cycle in a family of bipartite graphs?
\end{prob}
Focusing on Problem \ref{prob1.3}, we pose a sufficient condition based on the spectral radius for a family of balanced bipartite graphs to contain a rainbow Hamilton cycle.
\begin{thm}\label{thm1.6}
    Let $\mathcal{G}=\{G_1, G_2, \ldots , G_{2n}\}$ be a family of balanced bipartite graphs on vertex set $[2n]$ and the bipartition $(X, Y),$ where $n \geq 2.$ If
    \begin{eqnarray*}
        \rho (G_i) \geq \rho(K_{1, n-1}\sqcup \widehat{K_{n-1, 1}})
    \end{eqnarray*}
    for every $i\in [2n],$ then $\mathcal{G}$ admits a rainbow Hamilton cycle unless $G_1=G_2=\cdots =G_{2n}\cong K_{1, n-1}\sqcup \widehat{K_{n-1, 1}}.$
\end{thm}

\section{Preliminaries}
In this section, we first give an important technique and several auxiliary results, all of which will be employed in our subsequent arguments.

In extremal set theory, the technique of shifting is one of the most essential and widely-used tools. The shifting operation on graphs, also referred to as the Kelmans operation (see, e.g., \cite{brouwer2}), allows us to focus on sets possessing a particular structural property. Let $G$ be a graph on vertex set $[n].$ Define the $(x, y)$-shift $S_{xy}$ as $S_{xy}(G) = \{S_{xy}(e) :e \in E(G)\},$ where
\begin{equation}
\nonumber
S_{xy}(e)= \left\{
\begin{array}{cl}
(e\setminus\{y\})\cup\{x\},         &\text{if $y\in e,$ $x\notin e$ and $(e\setminus\{y\})\cup\{x\}\notin E(G)$}; \\
e,                     & \text{otherwise.}
\end{array} \right.
\end{equation}
If $S_{xy}(G) = G$ for every pair $(x, y)$ satisfying $x < y,$ then $G$ is said to be shifted, denoted by $S(G).$ By the definition of $(x, y)$-shift, we have $|E(G)|=|E(S_{xy}(G))|.$ In other words, $(x, y)$-shift does not change the number of edges. Moreover, one naturally poses the problem whether it affects the spectral radius of a graph. In 2009, Csikv\'{a}ri\cite{csikvari2009conjecture} proved that the shifting operation does not decrease the spectral radius of a graph.
\begin{lemma}[Csikv\'{a}ri\cite{csikvari2009conjecture}]\label{lem2.1}
    Let $x, y$ be two vertices of $G.$ Then $\rho(S_{xy}(G))\geq \rho(G).$
\end{lemma}

Guo et al.\cite{guo2023spectral} determined changes of the spectral radius after $(x, y)$-shift operation in connected graphs.
\begin{lemma}[Guo et al.\cite{guo2023spectral}]\label{lem2.2}
    Let $G$ be a connected graph on vertex set $[n].$ Let $x, y$ be two vertices of $G.$ Then $\rho(S_{xy}(G))> \rho(G)$ unless $G\cong S_{xy}(G).$
\end{lemma}
For a bipartite graph, to ensure that the graph remains bipartite after the $(x, y)$-shift, we must choose the vertices $x$ and $y$ from the same part. Let $G$ be a bipartite graph on vertex set $[n].$ If $S_{xy}(G) = G$ for every pair $(x, y)$ in the same part of $G$ satisfying $x < y,$ then $G$ is said to be bi-shifted. The following observation shows a basic property of $(x, y)$-shift for a bipartite graph $G.$

\begin{observation}\label{obs::1}
Let $G$ be a bipartite graph on vertex set $[n]$ with the partition $(X, Y).$  If $G$ is bi-shifted, then for any $\{x_1, x_2\}\subseteq X$ and $\{y_1, y_2\} \subseteq Y$ such that $x_1 \leq x_2$ and $y_1\leq y_2,$ $\{x_2, y_2\} \in E(G)$ implies $\{x_1, y_1\} \in E(G).$
\end{observation}
Iterating the $(x, y)$-shift for all pairs $(x, y)$ satisfying $x < y$ will eventually produce a shifted graph (see \cite{frankl1987} and \cite{frankl1995shifting}). Similarly, for a bipartite graph $G$, it can be verified that by repeatedly applying the shifting operation in the same part of $G$, one can obtain a bi-shifted graph.

\vspace{3mm}
Let $S_{xy}(\mathcal{G})=\{S_{xy}(G_1), S_{xy}(G_2), \ldots, S_{xy}(G_n)\}$ and $S(\mathcal{G})=\{S(G_1), S(G_2), \ldots , S(G_n)\}.$
\begin{lemma}[He, Li and Feng\cite{he2024spectral}]\label{lem2.3}
    Given a family of graphs $\mathcal{G}=\{G_1, G_2, \ldots,$ $ G_{n-1}\}$ on vertex set $[n].$ Let $x<y$ be two vertices in $[n].$ If $S_{xy}(\mathcal{G})$ admits a rainbow Hamilton path, then so does $\mathcal{G}.$
\end{lemma}
\begin{lemma}[Zhang and van Dam\cite{zhang2025}]\label{lem2.4}
    Let $\mathcal{G}=\{G_1, G_2, \ldots, G_n\}$ be a family of graphs on vertex set $[n].$ If $S(\mathcal{G})$ admits a rainbow Hamilton cycle, then so does $\mathcal{G}.$
\end{lemma}

We will use the following spectral inequality for bipartite graphs, which is a direct corollary of a result of Nosal \cite{nosal1970eigenvalues}. (See also \cite{bhattacharya2008first}.)
\begin{lemma}[Nosal \cite{nosal1970eigenvalues}, Bhattacharya, Friedland and Peled \cite{bhattacharya2008first}]\label{lem2.5}
Let $G$ be a bipartite graph. Then $$\rho(G) \leq \sqrt{|E(G)|}.$$
\end{lemma}
Let $A$ be a symmetric real matrix whose rows and columns are indexed by $X=[n].$ Given a partition $\Pi:$ $X=X_1\cup X_2\cup \dots \cup X_m,$ the matrix $A$ is denoted by
\begin{eqnarray*}
A=\begin{bmatrix}
   A_{11}&\dots&A_{1m}\\\vdots &\ddots&\vdots\\A_{m1}&\dots&A_{mm}
\end{bmatrix},
\end{eqnarray*}
where $A_{ij}$ is the submatrix of $A$ with respect to rows in $X_i$ and columns in $X_j.$ Let $B_{\Pi}$ be a matrix of order $m$ whose $(i,j)$-entry equals the average row sum of $A_{ij}.$ Then $B_{\Pi}$ is called a quotient matrix of $A$ corresponding to this partition. If the row sum of each block $A_{ij}$ is constant, then the partition $\Pi$ is equitable.
\begin{lemma}[Brouwer and Haemers\cite{brouwer2}, Godsil and Royle\cite{Godsil}]\label{lem2.6} Let $A$ be a real symmetric matrix and $\rho(A)$ be its largest eigenvalue. Let $B_{\Pi}$ be an equitable quotient matrix of $A$. Then the eigenvalues of $B_{\Pi}$ are also eigenvalues of $A$. Furthermore, if $A$ is nonnegative and irreducible, then $\rho(A)=\rho(B_{\Pi}).$
\end{lemma}

\begin{lemma}\label{lem2.8}
  (i) For $1\leq k\leq n-2,$ we have $\rho(Q^k_n)<\rho(Q^0_n)$ and $\rho(T^k_n)<\rho(T^0_n).$\\
  (ii) For $2\leq k\leq n-2,$ we have $\rho(B^k_n)<\rho(B^1_n).$
\end{lemma}
\begin{proof}
(i) By calculation, we have $\rho(Q^0_n)=\sqrt{n(n-1)}$ and $\rho(T^0_n)=n-1.$ Combining Lemma \ref{lem2.5} and $1\leq k\leq n-2,$ we have
\begin{eqnarray*}
\rho(Q^k_n) \leq \sqrt{|E(Q^k_n)|} &=& \sqrt{nk+(n-k)(n-k-1)} \\
&=&\sqrt{k^2-(n-1)k+n(n-1)}\\
&<& \sqrt{n(n-1)}=\rho(Q^0_n).
\end{eqnarray*}

For $1\leq k\leq n-3,$ by Lemma \ref{lem2.5}, we have
\begin{eqnarray*}
\rho(T^k_n) \leq \sqrt{|E(T^k_n)|}
                &=&\sqrt{nk+(n-k-1)^2} \\
                &=&\sqrt{k^2-(n-2)k+(n-1)^2}\\
                &<& n-1=\rho(T^0_n).
\end{eqnarray*}
For $k=n-2,$ the adjacent matrix of $T^{n-2}_n=K_{n-2, 1}\sqcup \widehat{K_{1, n-1}}$ has the equitable quotient matrix
\begin{eqnarray*}
B_{\Pi}=\begin{bmatrix}
   0&0&1&n-1\\0&0&1&0\\n-2&1&0&0\\n-2&0&0&0
\end{bmatrix}.
\end{eqnarray*}
According to Lemma \ref{lem2.6}, we obtain that $\rho(B_{\Pi})=\rho(T^{n-2}_n).$ By computation, the characteristic polynomial of $B_{\Pi}$ is
\begin{eqnarray*}
P(x)=x^4-(n-1)^2x^2+(n-1)(n-2).
\end{eqnarray*}
Then we have
\begin{eqnarray*}
    P(n-1)=(n-1)^4-(n-1)^4+(n-1)(n-2)=(n-1)(n-2)>0.
\end{eqnarray*}
It is straightforward to check that if $x>n-1,$ then $P(x)>P(n-1).$ Thus, it follows that $\rho(T^{n-2}_n)<\rho(T^0_n).$

(ii) Note that $Q^0_n$ is a proper subgraph of $B^1_n.$ Then we have $\rho(Q^0_n)\leq
\rho(B^1_n).$ For $2\leq k\leq n-2,$ by Lemma \ref{lem2.5}, we have
\begin{eqnarray*}
\rho(B^k_n) \leq \sqrt{|E(B^k_n)|}
                &=&\sqrt{nk+(n-k)^2} \\
                &=&\sqrt{k^2-nk+n^2}\\
                &\leq& \sqrt{n^2-2n+4}\\
                &<&\sqrt{n(n-1)}\\
                &=&\rho(Q^0_n)\leq \rho(B^1_n).
\end{eqnarray*}
The result follows.
\end{proof}
  \begin{lemma}\label{lem2.8.0}
Let $p\geq q\geq 0$ be integers. Let $G\cong G_1\cup pK_1$ and $G'$ be a subgraph of $G,$ where $G_1$ is connected. If $\rho(G')=\rho(G),$ then $G'\cong G_1\cup qK_1.$
\end{lemma}
\begin{proof}
 Suppose that $G'_1$ is a component of $G'$ such that $\rho(G'_1)=\rho(G').$ Then $G'_1$ is a subgraph of $G.$ Furthermore, $G'_1$ is a subgraph of $G_1.$ Note that $\rho(G)=\rho(G_1).$ Then $\rho(G'_1)=\rho(G_1).$ Since $G_1$ is connected, $G'_1\cong G_1.$ Hence $G'\cong G_1\cup qK_1.$
\end{proof}
\begin{lemma}\label{lem2.9}
Let $r$ be a positive integer and $\mathcal{G}=\{G_1, G_2, \ldots, G_{n-1}\}$ be a family of graphs on vertex set $[n].$ Suppose that $P_{r+1}=u_1 u_2\cdots u_r u_{r+1}$ is a longest rainbow path of $\mathcal{G}$ such that $u_i u_{i+1}\in E(G_i)$ for every $i\in[r].$ If $d_{G_{r+1}}(u_{r+1})> 0,$ then $N_{G_{r+1}}(u_{r+1})\subset V(P_{r+1}).$
\end{lemma}
\begin{proof}
Suppose to the contrary that there exists a vertex $u_{r+2}\notin V(P_{r+1})$ such that $u_{r+1}u_{r+2}\in E(G_{r+1})$. Then $P_{r+2}=u_1 u_2 \cdots u_{r+1}u_{r+2}$ is a longer rainbow path of $\mathcal{G},$ which contradicts the choice of $P_{r+1}.$
\end{proof}

\section{Proof of Theorem \ref{thm1.4}}
Recall that $Q^0_n= K_{n,n-1}\cup K_1.$ Before presenting our proof, we show some crucial lemmas.
\begin{lemma}\label{lem3.1}
Let $G$ be a balanced bipartite graph on vertex set $[2n]$ and the bipartition $(X, Y),$ where $n\geq 2.$ Let $\{u, v\} \subseteq X$ or $\{u, v\} \subseteq Y$ with $u<v.$ If $\rho (G) \geq \rho(Q^0_n)$ and $S_{uv}(G)\cong Q^0_n,$ then $G\cong Q^0_n.$
\end{lemma}
\begin{proof}
Note that $S_{uv}(G)\cong Q^0_n.$ Let $w$ be the isolated vertex in $S_{uv}(G).$ Without loss of generality, we assume that $w\in X.$ Denote $X_1=X\setminus\{w\}.$ If $\{u, v\} \subseteq X_1,$ then by the definition of $(u, v)$-shift, we have $G=S_{uv}(G)\cong Q^0_n.$ If $u=w$ or $v=w,$ it follows from $u<v$ that $u\in X_1$ and $v=w\in X.$ We claim that $d_G(u)=0$ or $d_G(v)=0.$ In fact, if $d_G(u)>0$ and $d_G(v)>0,$ then $G$ is connected and $G \ncong S_{uv}(G).$ By Lemma \ref{lem2.2}, we have $\rho(S_{uv}(G)) >\rho(G).$ Note that $S_{uv}(G)\cong Q^0_n.$ Then
\begin{eqnarray*}
\rho(Q^0_n)=\rho(S_{uv}(G)) >\rho(G)\geq \rho(Q^0_n),
\end{eqnarray*}
a contradiction. Hence $d_G(u)=0$ or $d_G(v)=0,$ and so $G$ is a subgraph of $Q^0_n.$ Note that $\rho(G)\geq \rho(Q^0_n).$ By Lemma \ref{lem2.8.0}, $G\cong Q^0_n.$ If $\{u, v\} \subseteq Y,$ then we have $G=S_{uv}(G)\cong Q^0_n.$
\end{proof}
By Lemmas \ref{lem2.1} and \ref{lem3.1}, we immediately obtain the following result.
\begin{cor}\label{cor3.1}
Let $G$ be a balanced bipartite graph on vertex set $[2n]$ and the bipartition $(X, Y),$ where $n\geq 2.$ If $\rho(G) \geq \rho(Q^0_n)$ and $S(G)\cong Q^0_n,$ then $G\cong Q^0_n.$
\end{cor}
Next we prove a technical lemma which is very important to our main result.
\begin{lemma}\label{lem3.2}
    Let $\mathcal{G}=\{G_1, G_2, \ldots , G_{2n-1}\}$ be a family of balanced bipartite graphs on vertex set $[2n]$ and the bipartition $(X, Y),$ where $n\geq 2.$ If $G_i\cong Q^0_n$ for every $i\in [2n-1],$ and there exist $p, q\in [2n-1]$ such that $G_p \neq G_q,$ then $\mathcal{G}$ admits a rainbow Hamilton path.
\end{lemma}
\begin{proof}
Without loss of generality, we assume that $p=1$ and $q=2n-1.$ Since $G_i\cong Q^0_n,$ there exists exactly one vertex of degree $0$ and all other vertices have degree either $n-1$ or $n$ in $G_i$ for every $i\in [2n-1].$ Let $u_1$ is the isolated vertex of $G_{2n-1}.$ Assume that $u_1\in X.$ Since $G_1\neq G_{2n-1}$ and $G_1\cong Q^0_n,$ we conclude that $d_{G_1}(u_1)\geq n-1 >0.$ Next we divide our proof into two cases.

\begin{case}
$n=2.$
\end{case}

Suppose that $X=\{u_1, u_3\}$ and $Y=\{u_2, u_4\}.$ Note that $d_{G_1}(u_1)\geq n-1=1.$ Then we assume that $u_1u_2\in E(G_1).$ Since $d_{G_3}(u_1)=0,$ $d_{G_3}(u_3)=2$ and $u_2 u_3, u_3 u_4\in E(G_3).$ If $d_{G_2}(u_2)= 0,$ then $u_3 u_4\in E(G_2).$ Therefore, $P=u_1 u_2 u_3 u_4$ is a rainbow Hamilton path of $\mathcal{G},$ where $u_1 u_2\in E(G_1),$ $ u_2 u_3\in E(G_3)$ and $u_3 u_4\in E(G_2).$ If $d_{G_2}(u_3)=0,$ then $u_4 u_1\in E(G_2).$ Hence
$P=u_4 u_1 u_2 u_3$ is a rainbow Hamilton path of $\mathcal{G},$ where $u_4u_1\in E(G_2),$ $u_1u_2\in E(G_1)$  and $u_2 u_3\in E(G_3).$  If $d_{G_2}(u_3)\neq 0$ and $d_{G_2}(u_2)\ne 0,$ then $u_2 u_3\in E(G_2).$ Hence $P=u_1 u_2 u_3 u_4$ is a rainbow Hamilton path of $\mathcal{G},$ where $u_1 u_2\in E(G_1),$ $u_2 u_3\in E(G_2)$ and $u_3 u_4\in E(G_3).$

\begin{case}
$n\geq3.$
\end{case}

Since $d_{G_1}(u_1)\geq n-1\geq 2,$ there exists a vertex $u_2$ such that $u_1u_2\in E(G_1)$ and $d_{G_2}(u_2)\geq n-1\geq 2.$ Then we can find an edge $u_2 u_3\in E(G_2)$ and $u_3\neq u_1.$ Without loss of generality, let $P_{r+1}=u_1 u_2\cdots u_r u_{r+1}$ be a longest rainbow path of $\mathcal{G}$ such that $u_i u_{i+1}\in E(G_i)$ for every $i\in[r],$ where $r\geq 2.$
\begin{subcase}
$r$ is odd.
\end{subcase}
In this case, $3\leq r\leq 2n-1.$ Recall that $u_1\in X.$ It follows that $u_{2i-1}\in X$ and $u_{2i}\in Y$ for every $i\in[\frac{r+1}{2}].$
Since $u _r\in X$, we have $N_{G_r}(u_r)\subseteq Y.$ Note that $d_{G_r}(u_r)\geq n-1$ and $|V(P_{r})\cap Y| = \frac{r-1}{2}.$ If $3\leq r\leq 2n-5,$ then we have
\begin{eqnarray*}
    |N_{G_{r}}(u_{r})\setminus V(P_{r})|\geq(n-1)-\frac{r-1}{2}\geq 2.
\end{eqnarray*}
Hence there exists a new vertex $u'_{r+1}\notin V(P_r)$ such that $u_{r}u'_{r+1}\in E(G_{r})$ and $d_{G_{r+1}}(u'_{r+1})\geq n-1.$ Let $P'_{r+1}=u_1 u_2\cdots u_r u'_{r+1}.$ Then we have
\begin{eqnarray*}
|N_{G_{r+1}}(u'_{r+1})\setminus V(P'_{r+1})|\geq (n-1)-\frac{r+1}{2}\geq 1.
\end{eqnarray*}
 This implies that there exists a vertex $u_{r+2}\notin V(P'_{r+1})$ such that $u'_{r+1}u_{r+2}\in E(G_{r+1}).$ Now we find that $P=u_1 u_2\cdots u_r u'_{r+1}u_{r+2}$ is a longer rainbow path of $\mathcal{G},$ which contradicts the choice of $P_{r+1}.$ Hence $r=2n-3$ or $r=2n-1.$ If $r=2n-1,$ then $P_{2n}=u_1 u_2\cdots u_{2n-1}u_{2n}$ is a rainbow Hamilton path of $\mathcal{G}.$ Next we consider $r=2n-3.$ Let $u_{2n-1}, u_{2n}\in [2n]\setminus V(P_{2n-2}).$ Without loss of generality, we assume that $u_{2n-1}\in X$ and $u_{2n}\in Y.$
\begin{claim}\label{claim::1}
$d_{G_{2n-2}}(u_{2n-2})=0$ or $d_{G_{2n-2}}(u_{2n-1})=0.$
\end{claim}
\begin{proof}
If $d_{G_{2n-2}}(u_{2n-2})=0,$ then the result follows. Next we consider $d_{G_{2n-2}}(u_{2n-2})>0.$ By Lemma \ref{lem2.9}, $N_{G_{2n-2}}(u_{2n-2})\subset V(P_{2n-2}).$ If $d_{G_{2n-2}}(u_{2n-2})=n,$ then $u_{2n-1}\in N_{G_{2n-2}}(u_{2n-2})$ and $u_{2n-1}\in V(P_{2n-2}),$ a contradiction. Hence $d_{G_{2n-2}}(u_{2n-2})=n-1.$  Since $u_{2n-1}\notin V(P_{2n-2}),$ we have $u_{2n-1}\notin N_{G_{2n-2}}(u_{2n-2}).$ By $G_{2n-2}\cong Q^0_n,$ $d_{G_{2n-2}}(u_{2n-1})=0.$
\end{proof}
By Claim \ref{claim::1}, we have $u_1 u_{2n}\in E(G_{2n-2}).$ Recall that $d_{G_{2n-1}}(u_1)=0$ and $G_{2n-1}\cong Q^0_n$. Then $d_{G_{2n-1}}(u_{2n-2})=n-1$ and $u_{2n-2} u_{2n-1}\in E(G_{2n-1}).$ Hence $P= u_{2n}u_1u_2\cdots u_{2n-3}u_{2n-2} u_{2n-1}$ is a rainbow Hamilton path of $\mathcal{G}.$

\begin{subcase}
$r$ is even.
\end{subcase}

In this case, $2\leq r\leq 2n-2.$ Recall that $u_1\in X.$ It follows that $u_{2i+1}\in X$ and $u_{2i}\in Y$ for every $i\in[\frac{r}{2}].$ If $2\leq r\leq 2n-6,$ then we have
\begin{eqnarray*}
    |N_{G_{r}}(u_{r})\setminus V(P_{r})|\geq (n-1)-\frac{r}{2}\geq 2.
\end{eqnarray*}
Hence there exists a vertex $u''_{r+1}\notin V(P_{r})$ such that $u_{r}u''_{r+1}\in E(G_{r})$ and $d_{G_{r+1}}(u''_{r+1})\geq n-1.$ Let $P''_{r+1}=u_1 u_2\cdots u_r u''_{r+1}.$ Then we have
\begin{eqnarray*}
|N_{G_{r+1}}(u''_{r+1})\setminus V(P''_{r+1})|\geq (n-1)-\frac{r}{2}\geq 2.
\end{eqnarray*}
This implies that there exists a vertex $u_{r+2}\notin V(P''_{r+1})$ such that $u''_{r+1}u_{r+2}\in E(G_{r+1}).$ Therefore, $P=u_1 u_2\cdots u_r u''_{r+1}u_{r+2}$ is a longer rainbow path of $\mathcal{G},$ which contradicts the choice of $P_{r+1}.$ Then $r\geq 2n-4.$ Suppose that $r=2n-2.$ Recall that $u_1, u_{2n-1}\in X$ and $d_{G_{2n-1}}(u_1)=0.$ Then $d_{G_{2n-1}}(u_{2n-1})=n.$ Therefore,
\begin{eqnarray*}
|N_{G_{2n-1}}(u_{2n-1})\cap V(P_{2n-1})|\leq n-1<n.
\end{eqnarray*}
This implies that $N_{G_{2n-1}}(u_{2n-1}) \nsubseteq V(P_{2n-2}),$ contradicting Lemma \ref{lem2.9}. Next we consider $r=2n-4.$ Let $u_{2n-2},$ $u_{2n-1}, u_{2n}\in [2n]\setminus V(P_{2n-3})$ such that $u_{2n-1}\in X$ and $u_{2n-2}, u_{2n}\in Y.$ We claim that $d_{G_{2n-3}}(u_{2n-3})=0.$ In fact, if $d_{G_{2n-3}}(u_{2n-3})\geq n-1,$ then we have
\begin{eqnarray*}
    |N_{G_{2n-3}}(u_{2n-3})\cap V(P_{2n-3})|\leq n-2<n-1.
\end{eqnarray*}
It follows that $N_{G_{2n-3}}(u_{2n-3})\nsubseteq V(P_{2n-3}),$ contrary to Lemma \ref{lem2.9}. Note that $G_{2n-3}\cong Q^0_n.$ Then there exists an edge $u_{2n-2}u_{2n-1}$ $\in E(G_{2n-3}).$ Recall that $u_1$ is the unique isolated vertex of $G_{2n-1}.$ By $u_{2n-3}\in X$ and $u_{2n-3}\neq u_1,$ we have $d_{G_{2n-1}}(u_{2n-3})=n.$ Hence there exist two edges $u_{2n-3}u_{2n-2},$ $u_{2n-3}u_{2n}\in E(G_{2n-1}).$ Note that $d_{G_{2n-2}}(u_{2n-2})\geq n-1$ or $d_{G_{2n-2}}(u_{2n})\geq n-1.$ Without loss of generality, assume that $d_{G_{2n-2}}(u_{2n})\geq n-1.$ Then we have $u_{2n-1}u_{2n}$ $\in E(G_{2n-2})$ or $u_{1}u_{2n}$ $\in E(G_{2n-2}).$ If $u_{2n-1}u_{2n}$ $\in E(G_{2n-2})$, then $P'=u_1u_2\cdots u_{2n-3}u_{2n-2} u_{2n-1}u_{2n}$ is a rainbow Hamilton path of $\mathcal{G}.$ If $u_{1}u_{2n}$ $\in E(G_{2n-2}),$ then $P''=u_{2n}u_1u_2\cdots u_{2n-3}u_{2n-2}u_{2n-1}$ is a rainbow Hamilton path of $\mathcal{G}.$
\setcounter{case}{0}
\end{proof}

Now we are in a position to present the proof of Theorem \ref{thm1.4}.
\vspace{1mm}

\medskip
\noindent  \textbf{Proof of Theorem \ref{thm1.4}.}
For $n=2,$ by Lemma \ref{lem2.5}, we have $\rho(Q^0_2)\leq \rho(G_i)\leq \sqrt{|E(G_i)|},$ which implies that $2=\rho^2(Q^0_2)\leq |E(G_i)|\leq 4.$ If $|E(G_i)|=2$ for each $i\in [3],$ then $G_i \cong Q^0_2$ or $K_2\cup K_2.$ Note that $\rho(G_i)\geq\rho(Q^0_2)>\rho(K_2\cup K_2).$ Then $G_i \cong Q^0_2.$ Suppose to the contrary that $\mathcal{G}$ has no rainbow Hamilton path. By Lemma \ref{lem3.2}, we have $G_1= G_2=G_3\cong Q^0_2.$ If there exists some $i\in[3]$ such that $|E(G_i)|\geq 3,$ then without loss of generality, we assume that $|E(G_1)|\geq 3,$ and hence $Q^0_2$ is a proper subgraph of $G_1.$ It suffices to prove that $\mathcal{G}$ has a rainbow Hamilton path. Note that $Q^0_2$ is a subgraph of $G_i$ for each $i\in[3].$ Let $G'_i\cong Q^0_2$ be a subgraph of $G_i,$ where $i=2,3.$ Since $|E(G_1)|\geq 3,$ there exist two subgraphs $G'_1$ and $G''_1$ of $G_1$ such that $G'_1\cong Q^0_2,$ $G''_1\cong Q^0_2$ and $G'_1\neq G''_1.$ Then $G'_1\neq G'_2$ or $G''_1\neq G'_2.$ By Lemma \ref{lem3.2}, $\{G'_1, G'_2, G'_3\}$ or $\{G''_1, G'_2, G'_3\}$ has a rainbow Hamilton path, which implies that so does $\mathcal{G}.$

Next we consider the case $n\geq 3.$ Assume that $\mathcal{G}=\{G_1, G_2, \ldots,$ $ G_{2n-1}\}$ has no rainbow Hamilton path. For convenience, we denote $S_i=S(G_i)$ for each $i\in [2n-1].$ According to Lemma \ref{lem2.3}, the bi-shifted family $\{S_1,$ $S_2,$ $\ldots,$ $S_{2n-1}\}$ has no rainbow Hamilton path. By Lemma \ref{lem2.1} and $\rho(G_i)\geq \rho(Q^0_n),$ we have
\begin{eqnarray}\label{eq::1}
    \rho(S_i)\geq \rho(G_i)\geq \rho(Q^0_n).
\end{eqnarray}
 We first prove that $S_i\cong Q^0_n$ for every $i\in[2n-1].$ Without loss of generality, we assume that $X=[n]$ and $Y=[2n]\setminus[n].$ Define $e_j=\{j, 2n-j+1\},$ where $1\leq j\leq n.$

\begin{claim}\label{claim::2}
$\{e_2, e_3, \ldots, e_{n-1}\}\subseteq E(S_i)$ for each $i\in [2n-1].$
\end{claim}
\begin{proof}
Suppose that Claim \ref{claim::2} does not hold. Then there exist some $2\leq p \leq n-1$ and some $q\in [2n-1]$ such that $e_p=\{p, 2n-p+1\} \notin E(S_q).$ Note that $S_q$ is bi-shifted. By Observation \ref{obs::1}, for every edge $\{i, j\}\in E(S_q),$ we have $1\leq i<p$ or $n+1\leq j <2n-p+1.$ Hence $S_q$ is a subgraph of $Q^{p-1}_n,$ which implies that $ \rho(S_q)\leq \rho(Q^{p-1}_n).$ According to Lemma \ref{lem2.8}, we have
\begin{eqnarray*}
    \rho(S_q)\leq \rho(Q^{p-1}_n)<\rho(Q^0_n),
\end{eqnarray*}
which contradicts (\ref{eq::1}).
\end{proof}

If $d_{S_i}(n)=0$ for each $i\in [2n-1],$ then $S_i$ is a subgraph of $Q^0_n.$ By (\ref{eq::1}) and Lemma \ref{lem2.8.0}, $S_i\cong Q^0_n$ for $i\in[2n-1].$ Similarly, if $d_{S_i}(2n)=0$ for each $i\in [2n-1],$ then $S_i\cong Q^0_n.$ Next we consider the case that there exist $p, q\in [2n-1]$ such that $d_{S_p}(n)>0$ and $d_{S_q}(2n)>0.$ Recall that $S_p$ and $S_q$ are bi-shifted. By Observation \ref{obs::1}, we can deduce that $e_n=\{n, n+1\}\in E(S_p)$ and $e_1=\{1, 2n\}\in E(S_q).$ Next we divide our proof into two cases.
\begin{case}\label{c1}
$p\neq q.$
\end{case}
Without loss of generality, we assume that $p=n$ and $q=1.$ Then $e_1\in E(S_1)$ and $e_n\in E(S_n).$ By Claim \ref{claim::2}, we can choose $e_j\in E(S_j)$ for $2\leq j\leq n-1.$ Note that $e_2=\{2, 2n-1\}\in E(S_{n+1})$ and $e_j=\{j, 2n-j+1\}\in E(S_{n+j})$ for $2\leq j\leq n-1.$ By Observation \ref{obs::1}, $\{1, 2n-1\}\in E(S_{n+1})$ and $\{j, 2n-j\}\in E(S_{n+j})$ for $2\leq j\leq n-1.$
Hence $\{S_1, S_2, \ldots, S_{2n-1}\}$ has a rainbow Hamilton path, a contradiction.
\begin{case}
$p=q.$
\end{case}
Without loss of generality, we assume that $p=q=1.$ We claim that $d_{S_i}(n)=d_{S_i}(2n)=0$ for every $2\leq i\leq 2n-1.$ In fact, if there exists some $2\leq j\leq 2n-1$ such that $d_{S_j}(n)>0$ or $d_{S_j}(2n)>0,$ then we have $d_{S_1}(n)>0$ and $d_{S_j}(2n)>0,$ or $d_{S_1}(2n)>0$ and $d_{S_j}(n)>0.$
By Case \ref{c1}, $\{S_1, S_2, \ldots, S_{2n-1}\}$ has a rainbow Hamilton path, a contradiction. Hence $\rho(S_i)\leq \rho(K_{n-1,n-1})<\rho(Q^0_n)$ for every $2\leq i\leq 2n-1,$ which contradicts (\ref{eq::1}).

Therefore, $S_i\cong Q^0_n$ for each $i\in [2n-1].$ Combining (\ref{eq::1}), we have $\rho(S_i)=\rho(G_i)=\rho(Q^0_n).$ It follows from Corollary \ref{cor3.1} that $G_i\cong Q^0_n.$ Suppose that there exist $p,q\in [2n-1]$ such that $G_p\neq G_q.$ By Lemma \ref{lem3.2}, $\mathcal{G}$ has a rainbow Hamilton path, a contradiction. Hence $G_1=G_2=\cdots=G_{2n-1} \cong Q^0_n.$
\setcounter{case}{0}
\hspace*{\fill}$\Box$

\section{Proof of theorem \ref{thm1.5}}
Recall that $T^0_n= K_{n-1,n-1}\cup K_1.$ In this section, we shall give the proof of Theorem \ref{thm1.5}.
\begin{lemma}\label{lem4.1}
Let $G$ be a nearly balanced bipartite graph on vertex set $[2n-1]$ and the bipartition $(X, Y).$ Let $\{u, v\} \subseteq X$ or $\{u, v\} \subseteq Y$ with $u<v.$ If $\rho (G) \geq \rho(T^0_n)$ and $S_{uv}(G)\cong T^0_n,$ then $G\cong T^0_n.$
\end{lemma}
\begin{proof}
Note that $S_{uv}(G)\cong T^0_n.$ Let $w$ be the vertex of degree $0$ in $S_{uv}(G).$ Without loss of generality, suppose that $w\in X.$ Denote $X_1=X\setminus\{w\}.$  If $\{u, v\} \subseteq X_1,$ then we have $G=S_{uv}(G)\cong T^0_n.$ If $u=w$ or $v=w,$ then we obtain that $u\in X_1$ and $v=w\in X$ for $u<v.$ We claim that $d_G(u)=0$ or $d_G(v)=0.$ Otherwise, we get $d_G(u)>0$ and $d_G(v)>0.$ Then $G$ is connected and $G\ncong S_{uv}(G).$ By Lemma \ref{lem2.2}, $\rho(S_{uv}(G)) >\rho(G).$ Recall that $S_{uv}(G)\cong T^0_n.$ Then we have
\begin{eqnarray*}
    \rho(T^0_n)=\rho(S_{uv}(G)) >\rho(G)\geq \rho(T^0_n),
\end{eqnarray*}
a contradiction. Hence $d_G(u)=0$ or $d_G(v)=0,$ which implies that $G$ is a subgraph of $T^0_n.$ Note that $\rho(G)\geq \rho(T^0_n).$ By Lemma \ref{lem2.8.0}, $G\cong T^0_n.$ If $\{u, v\} \subseteq Y,$ then we deduce that $G=S_{uv}(G)\cong T^0_n.$
\end{proof}
Based on Lemmas \ref{lem2.1} and \ref{lem4.1}, we can directly obtain the following result.
\begin{cor}\label{cor4.1}
Let $G$ be a nearly balanced bipartite graph on vertex set $[2n-1]$ and the bipartition $(X, Y).$ If $\rho (G) \geq \rho(T^0_n)$ and $S(G)\cong T^0_n,$ then $G\cong T^0_n.$
\end{cor}
We next establish a technical lemma that plays a crucial role in our main result.
\begin{lemma}\label{lem4.2}
Let $\mathcal{G}=\{G_1, G_2, \ldots , G_{2n-2}\}$ be a family of nearly balanced bipartite graphs on vertex set $[2n-1]$ and the bipartition $(X, Y).$ If $G_i\cong T^0_n$ for every $i\in [2n-2],$ and there exist $p, q\in [2n-2]$ such that $G_p \neq G_q,$ then $\mathcal{G}$ admits a rainbow Hamilton path.
\end{lemma}
\begin{proof}
Without loss of generality, we assume that $p=1,$ $q=2n-2,$ and $|X|=n.$ Since $G_i\cong T^0_n$ and $|X|=n,$ there exists the unique isolated vertex in $X$ and all other vertices have degree $n-1$ in $G_i$ for $i\in [2n-2].$ Let $u_1$ be the isolated vertex of $G_{2n-2}.$ Clearly, $u_1\in X.$ Since $G_1\neq G_{2n-2}$ and $G_1\cong T^0_n,$ we can obtain that $d_{G_1}(u_1)=n-1>0.$ Next we divide the following proof into two cases.
\begin{case}
    $n=2.$
\end{case}
Suppose that $X=\{u_1, u_3\}$ and $Y=\{u_2\}.$ Note that $d_{G_1}(u_1)=n-1=1.$ Then we have $u_1u_2\in E(G_1).$ Since $d_{G_2}(u_1)=0,$ we deduce that $d_{G_2}(u_3)=1,$ and hence $u_2 u_3\in E(G_2).$ Then $P=u_1 u_2 u_3$ is a rainbow Hamilton path of $\mathcal{G},$ where $u_1 u_2\in E(G_1)$ and $u_2 u_3\in E(G_2).$
\begin{case}
    $n\geq 3.$
\end{case}
Since $d_{G_1}(u_1)=n-1\geq 2,$ there exists a vertex $u_2$ such that $u_1 u_2\in E(G_1)$ and $d_{G_2}(u_2)=n-1\geq 2.$ Then we can find an edge $u_2 u_3\in E(G_2)$ such that $u_3\neq u_1.$ Without loss of generality, let $P_{r+1}=u_1 u_2\cdots u_r u_{r+1}$ be a longest rainbow path of $\mathcal{G}$ such that $u_i u_{i+1}\in E(G_i)$ for every $i\in[r],$ where $r\geq 2.$
\begin{subcase}
    $r$ is odd.
\end{subcase}
In this case, $3\leq r \leq 2n-3.$ By $u_1\in X,$ we have $u_{2i-1}\in X$ and $u_{2i}\in Y$ for $i\in[\frac{r+1}{2}].$ Then $N_{G_r}(u_r)\subseteq Y.$ Notice that $d_{G_r}(u_r)=n-1$ and $|V(P_r)\cap Y| = \frac{r-1}{2}.$ If $3\leq r\leq 2n-5,$ then we have
\begin{eqnarray*}
    |N_{G_{r}}(u_{r})\setminus V(P_r)|=(n-1)-\frac{r-1}{2}\geq 2.
\end{eqnarray*}
Hence there exists some vertex $u'_{r+1}\notin V(P_r)$ such that $u_{r}u'_{r+1}\in E(G_{r})$ and $d_{G_{r+1}}(u'_{r+1})= n-1.$ Similarly, we have
\begin{eqnarray*}
|N_{G_{r+1}}(u'_{r+1})\setminus V(P_r)|\geq (n-1)-\frac{r+1}{2}\geq 1.
\end{eqnarray*}
Then there exists some vertex $u_{r+2}\notin V(P_r)$ such that $u'_{r+1}u_{r+2}\in E(G_{r+1}).$ Consequently, $P=u_1 u_2\cdots u_r u'_{r+1}u_{r+2}$ is a longer rainbow path of $\mathcal{G},$ which contradicts the choice of $P_{r+1}.$
Then we conclude that $r=2n-3,$ which implies that $|V(P_{2n-2})|=2n-2.$ Let $u_{2n-1}\in X$ such that $u_{2n-1}\notin V(P_{2n-2}).$ Since $G_{2n-2}\cong T^0_n$ and $d_{G_{2n-2}}(u_1)=0,$ we have $d_{G_{2n-2}}(u_{2n-2})=n-1.$ It follows that
\begin{eqnarray*}
    |N_{G_{2n-2}}(u_{2n-2})\cap V(P_{2n-2})|\leq n-2<n-1.
\end{eqnarray*}
One can deduce that $N_{G_{2n-2}}(u_{2n-2})\nsubseteq V(P_{2n-2}),$  which contradicts Lemma \ref{lem2.9}.
\begin{subcase}
    $r$ is even.
\end{subcase}
In this case, $2\leq r\leq 2n-2.$ By $u_1\in X,$ we have $u_{2i+1}\in X$ and $u_{2i}\in Y$ for every $i\in[\frac{r}{2}].$
If $2\leq r\leq 2n-6,$ then we have
\begin{eqnarray*}
    |N_{G_{r}}(u_{r})\setminus V(P_r)|\geq (n-1)-\frac{r}{2}\geq 2.
\end{eqnarray*}
Hence there exists some vertex $u''_{r+1}$ such that $u_{r}u''_{r+1}\in E(G_{r})$ and $d_{G_{r+1}}(u''_{r+1})= n-1.$ Similarly, we have
\begin{eqnarray*}
    |N_{G_{r+1}}(u''_{r+1})\setminus V(P_r)|\geq (n-1)-\frac{r}{2}\geq 2.
\end{eqnarray*}
Then there exists some vertex $u_{r+2}\notin V(P_{r+1})$ such that $u''_{r+1}u_{r+2}\in E(G_{r+1}).$ Furthermore, we can obtain that $P=u_1 u_2\cdots u_r u''_{r+1}u_{r+2}$ is a longer rainbow path of $\mathcal{G},$ contradicting the choice of $P_{r+1}.$
Then $r\geq2n-4.$ If $r=2n-2,$ then $P_{2n-2}=u_1 u_2\cdots u_{2n-2}u_{2n-1}$ is a rainbow Hamilton path of $\mathcal{G}.$ If $r=2n-4,$ then $|V(P_{2n-3})|=2n-3.$ Let $u_{2n-2},u_{2n-1}\in [2n-1]\setminus V(P_{2n-3}).$ Without loss of generality, suppose that $u_{2n-1}\in X$ and $u_{2n-2}\in Y.$ We claim that $d_{G_{2n-3}}(u_{2n-3})=0.$ Otherwise, we have $d_{G_{2n-3}}(u_{2n-3})=n-1.$ Since $u_{2n-3}\in X,$ we can obtain that
\begin{eqnarray*}
    |N_{G_{2n-3}}(u_{2n-3})\cap V(P_{2n-3})|\leq|Y\cap V(P_{2n-3})|=n-2<n-1.
\end{eqnarray*}
This implies that $N_{G_{2n-3}}(u_{2n-3})\nsubseteq V(P_{2n-3}),$ contrary to Lemma \ref{lem2.9}. Then there is an edge $u_{2n-2}u_{2n-1}\in E(G_{2n-3}).$ Since $d_{G_{2n-2}}(u_1)=0$ and $G_{2n-2}\cong T^0_n,$ there is an edge $u_{2n-3}u_{2n-2}\in E(G_{2n-2}).$ Hence $P=u_1u_2\cdots u_{2n-3}u_{2n-2}u_{2n-1}$ is a rainbow Hamilton path of $\mathcal{G}.$
\setcounter{case}{0}
\end{proof}
Now we are in a position to present the proof of Theorem \ref{thm1.5}.

\medskip
\noindent  \textbf{Proof of Theorem \ref{thm1.5}.}
For $n=2$ and $i\in [2],$ if $\rho(G_i)\geq\rho(T^0_2),$ then we have $G_i\cong K_2\cup K_1=T^0_2$ or $G_i\cong K_{2,1}.$ Hence $T^0_2$ is a subgraph of $G_i.$ Without loss of generality, suppose that $X=\{u_1, u_3\},$ $Y=\{u_2\}$ and $|E(G_1)|\geq |E(G_2)|.$ If $|E(G_1)|=1,$ then $G_1\cong G_2\cong T^0_2.$ Suppose that $\mathcal{G}$ has no rainbow Hamilton path. By Lemma \ref{lem4.2}, $G_1=G_2\cong T^0_2.$ If $|E(G_1)|=2,$ then $G_1\cong K_{2,1}.$  We arbitrarily choose an edge, say $u_1u_2,$ in $G_2.$ Since $u_2u_3\in E(G_1),$ $P=u_1u_2u_3$ is a rainbow Hamilton path of $\mathcal{G}.$

For $n\geq 3,$ we suppose to the contrary that $\mathcal{G}=\{G_1, G_2, \ldots,$ $ G_{2n-2}\}$ has no rainbow Hamilton path. For convenience, we denote $S_i=S(G_i)$ for every $ i\in [2n-2].$ It follows from Lemma \ref{lem2.3} that the bi-shifted family $\{S_1,$ $S_2,$ $\ldots,$ $S_{2n-2}\}$ has no rainbow Hamilton path. By Lemma \ref{lem2.1} and $\rho(G_i)\geq \rho(T^0_n)$, we have
\begin{eqnarray}\label{eq::2}
    \rho(S_i)\geq \rho(G_i)\geq \rho(T^0_n).
\end{eqnarray}
Next we show that $S_i\cong T^0_n$ for every $i\in[2n-2].$ Without loss of generality, we suppose that $X=[n-1]$ and $Y=[2n-1]\setminus[n-1].$ Define $e_j=\{j, 2n-j\},$ where $1\leq j\leq n-1.$
\begin{claim}\label{claim::4}
$\{e_2, e_3, \ldots, e_{n-1}\}\subseteq E(S_i)$ for each $i\in [2n-2].$
\end{claim}
\begin{proof}
Assume that Claim \ref{claim::4} does not hold. Then there exist some $2\leq p \leq n-1$ and some $q\in [2n-1]$ such that $e_p=\{p, 2n-p\} \notin E(S_q).$ Since $S_q$ is bi-shifted, by Observation \ref{obs::1}, we can obtain that $1\leq i<p$ or $n\leq j <2n-p$ for each edge $\{i, j\}\in E(S_q).$ This implies that $S_q$ is a subgraph of $T^{p-1}_n.$ Hence $\rho(S_q)\leq \rho(T^{p-1}_n).$ By Lemma \ref{lem2.8}, we obtain that
\begin{eqnarray*}
    \rho(S_q)\leq \rho(T^{p-1}_n)<\rho(T^0_n),
\end{eqnarray*}
which contradicts (\ref{eq::2}).
\end{proof}
If $d_{S_i}(2n-1)=0$ for each $i\in [2n-2],$ then $S_i$ is a subgraph of $T^0_n.$ By (\ref{eq::2}) and Lemma \ref{lem2.8.0}, we have $S_i\cong T^0_n.$ Now we assume that there exists some $p\in [2n-2]$ such that $d_{S_p}(2n-1)>0.$ Recall that $S_p$ are bi-shifted. By Observation \ref{obs::1}, we conclude that $e_1=\{1, 2n-1\}\in E(S_p).$ Without loss of generality, we assume that $p=1.$ It follows from Claim \ref{claim::4} that $e_j\in E(S_j)$ for $2\leq j\leq n-1.$ Note that $e_2=\{2, 2n-2\}\in E(S_{n})$ and $e_j=\{j, 2n-j\}\in E(S_{n+j-1})$ for $2\leq j\leq n-1.$ By Observation \ref{obs::1}, $\{1, 2n-2\}\in E(S_{n})$ and $\{j, 2n-j-1\}\in E(S_{n+j-1}).$ Then $\{S_1, S_2, \ldots, S_{2n-2}\}$ has a rainbow Hamilton path, a contradiction. Hence $S_i\cong T^0_n$ for $i\in [2n-2].$

Based on (\ref{eq::2}), we have $\rho(S_i)=\rho(G_i)=\rho(T^0_n).$ According to Corollary \ref{cor4.1}, we can obtain that $G_i\cong T^0_n.$ Assume that there exist $p,q\in [2n-1]$ such that $G_p\neq G_q.$ By Lemma \ref{lem4.2}, $\mathcal{G}$ has a rainbow Hamilton path, a contradiction. Hence $G_1=G_2=\cdots=G_{2n-2}$ and $G_1\cong T^0_n.$
\hspace*{\fill}$\Box$

\section{Proof of theorem \ref{thm1.6}}
Recall that $B^1_n= K_{1, n-1}\sqcup \widehat{K_{n-1, 1}}.$ In this section, we present the proof of Theorem \ref{thm1.6}.
\begin{lemma}\label{lem5.1}
Let $G$ be a balanced bipartite graph on vertex set $[2n]$ and the bipartition $(X, Y),$ where $n\geq 4.$ Suppose that $\{u, v\} \subseteq X$ or $\{u, v\} \subseteq Y$ with $u<v.$ If $\rho (G) \geq \rho(B^1_n)$ and $S_{uv}(G)\cong B^1_n,$ then $G\cong B^1_n.$
\end{lemma}
\begin{proof}
Note that $S_{uv}(G)\cong B^1_n.$ Denote by $w$ the vertex of degree $1$ in $S_{uv}(G).$ Without loss of generality, we assume that $w\in X.$ Let $X_1=X\setminus \{w\}.$ If $\{u, v\} \subseteq X_1,$ then we have $G=S_{uv}(G)\cong B^1_n.$ If $u=w$ or $v=w,$ then by $u<v,$ we have $u\in X_1$ and $v=w.$ Next we show that $d_G(u)>0$ and $d_G(v)>0.$ Suppose to the contrary that $d_G(u)=0$ or $d_G(v)=0.$ Then $d_{S_{uv}(G)}(u)=0$ or $d_{S_{uv}(G)}(v)=0,$ contradicting $S_{uv}(G)\cong B^1_n.$ Then $d_G(u)>0$ and $d_G(v)>0,$ and hence $G$ is connected. Note that
\begin{eqnarray*}
\rho(G)\geq \rho(B^1_n)=\rho(S_{uv}(G)).
\end{eqnarray*}
By Lemma \ref{lem2.2}, we have $\rho(G)=\rho(S_{uv}(G))$ and $G\cong S_{uv}(G)\cong B^1_n.$ If $\{u, v\} \subseteq Y,$ then we can obtain that $G= S_{uv}(G)\cong B^1_n.$
\end{proof}
Combining Lemmas \ref{lem2.1} and \ref{lem5.1}, we have the following result.
\begin{cor}\label{cor5.1}
Let $G$ be a balanced bipartite graph on vertex set $[2n]$ and the bipartition $(X, Y),$ where $n\geq 4.$ If $\rho (G) \geq \rho(B^1_n)$ and $S(G)\cong B^1_n,$ then $G\cong B^1_n.$
\end{cor}
\begin{lemma}\label{lem5.2}
Let $\mathcal{G}=\{G_1, G_2, \ldots , G_{2n}\}$ be a family of balanced bipartite graphs on vertex set $[2n]$ and the bipartition $(X, Y),$ where $n \geq 2.$ If $G_i\cong B^1_n$ for every $i\in [2n],$ and there exist $p, q\in [2n]$ such that $G_p \neq G_q,$ then $\mathcal{G}$ admits a rainbow Hamilton cycle.
\end{lemma}
\begin{proof}
Suppose that $p=2n.$ Since $Q^0_n$ is a subgraph of $B^1_n,$ we have $\rho(Q^0_n)\leq \rho(B^1_n).$ By Theorem \ref{thm1.4}, $\{G_1, G_2, \ldots , G_{2n-1}\}$ has a rainbow Hamilton path, denoted as $P=u_1u_2\cdots u_{2n-1}u_{2n}.$ Suppose that $u_{2i-1}\in X$ and $u_{2i}\in Y$ for each $i\in[n].$ Let $u_iu_{i+1}\in E(G_i)$ for $ i\in [2n-1]$ by reordering the graphs $G_1, G_2, \ldots , G_{2n-1}$ if necessary. If $u_1u_{2n}\in E(G_{2n}),$ then $C=u_1u_2\cdots u_{2n-1}u_{2n}u_1$ is a rainbow Hamilton cycle of $\mathcal{G}.$ Next we assume that $u_1u_{2n}\notin E(G_{2n}).$

For $n=2,$ we have $u_iu_{i+1}\in E(G_{4}),$ where $i\in[3].$ Since $G_q\neq G_4,$ it is clear that $u_1u_4\in E(G_q).$ We choose $u_qu_{q+1}\in E(G_4).$ Then $C=u_1u_2u_3u_4u_1$ is a rainbow Hamilton cycle of $\mathcal{G}.$

For $n\geq 3,$ since $G_{2n} \cong B^1_n$ and $u_1u_{2n}\notin E(G_{2n}),$ we have $d_{G_{2n}}(u_1)=1$ or $d_{G_{2n}}(u_{2n})=1.$ Without loss of generality, we assume that $d_{G_{2n}}(u_1)=1$ and $u_1u_s\in E(G_{2n}),$ where $2\leq s\leq 2n-2.$ Next we consider two cases according to the different values of $s$.
\begin{case}
$s=2.$
\end{case}
\begin{claim}\label{claim::6}
If there exists some $j\in [2n-1]$ such that $u_1u_{2n}\in E(G_j),$ then $\mathcal{G}$ admits a rainbow Hamilton cycle.
\end{claim}
\begin{proof}
Since $G_{2n}\cong B^1_n$ and $u_1u_2\in E(G_{2n}),$ we have $u_ju_{j+1}\in E(G_{2n}).$ Note that $u_1u_{2n}\in E(G_j).$ Then $C=u_1u_2\cdots u_{2n-1}u_{2n}u_1$ is a rainbow Hamilton cycle of $\mathcal{G}.$
\end{proof}
Based on Claim \ref{claim::6}, we only consider the case $u_1u_{2n}\notin E(G_i)$ for each $i\in [2n].$ This means that $d_{G_i}(u_1)=1$ or $d_{G_i}(u_{2n})=1.$
\begin{subcase}
    $d_{G_i}(u_1)=1$ for each $i\in [2n-1].$
\end{subcase}
Recall that $G_{2n}\neq G_q.$ Suppose that $u_1u_k\in E(G_q)$ and $u_k\neq u_2.$ Let $G'_i$ be the graph induced on $V(G_i)\setminus \{u_1\}$ in $G_i.$ Then we have $G'_1=G'_2=\cdots=G'_{2n}\cong K_{n.n-1}.$ Hence we can find a rainbow $(u_2,u_k)$-path $P'$ with $2n-1$ vertices in $\{G_1, G_2, \ldots , G_{q-1}, G_{q+1}, \ldots G_{2n-1}\}$ such that $u_1\notin V(P').$ By choosing edges $u_1u_2$ in $G_{2n}$ and $u_1u_k$ in $G_q$, we can obtain a rainbow Hamilton cycle in $\mathcal{G}.$
\begin{subcase}
    $d_{G_j}(u_1)\neq 1$ for some $j\in[2n-1].$
\end{subcase}
Note that $G_j\neq G_{2n}.$ Without loss of generality, we assume that $j=q.$ Then $d_{G_q}(u_1)=n-1$ and $d_{G_q}(u_{2n})=1.$ Now we show that $C=u_1u_2u_3u_{2n}u_{2n-1}\cdots u_4u_1$ is a rainbow Hamilton cycle of $\mathcal{G}.$ Recall that $d_{G_i}(u_1)=1$ or $d_{G_i}(u_{2n})=1$ for each $i\in [2n].$ If $2\leq q\leq 2n-2,$ then we have $u_qu_{q+1}\in E(G_3).$ Note that $u_1u_4\in E(G_q)$ and $u_3u_{2n}\in E(G_{2n}).$ Then $C$ is a rainbow Hamilton cycle of $\mathcal{G}.$ Suppose that $q=1$ or $q=2n-1.$ Based on the proof above, we only consider the case $d_{G_i}(u_1)=1$ for $2\leq i\leq 2n-2.$ We choose $u_3u_{2n}\in E(G_3),$ $u_1u_4\in E(G_q)$ and $u_qu_{q+1}\in E(G_{2n}).$ Hence $C$ is a rainbow Hamilton cycle of $\mathcal{G}.$

\begin{case}
   $4\leq s\leq 2n-2.$
\end{case}
Recall that $d_{G_{2n}}(u_1)=1$ and $G_{i}\cong B^1_n$ for each $i\in [2n].$ Then we have $u_{s-1}u_s\in E(G_{2n}).$ If $u_1u_{2n}\in E(G_{s-1}),$ then $C=u_1u_2\cdots u_{2n-1}u_{2n}u_1$ is a rainbow Hamilton cycle of $\mathcal{G}.$ Consequently, we assume that $u_1u_{2n}\notin E(G_{s-1}),$ which implies that $d_{G_{s-1}}(u_1)=1$ or $d_{G_{s-1}}(u_{2n})=1.$ Next we show that the cycle $C=u_1u_2\cdots u_{s-1}u_{2n}u_{2n-1}\cdots u_su_1$ is a rainbow Hamilton cycle of $\mathcal{G}.$ If $u_{s-1}u_{2n}\in E(G_{s-1}),$ then we choose $u_1u_s\in E(G_{2n}),$ and hence $C$ is a rainbow Hamilton cycle of $\mathcal{G}.$ If $u_{s-1}u_{2n}\notin E(G_{s-1}),$ then $d_{G_{s-1}}(u_{2n})=1,$ which implies that $u_1u_s\in E(G_{s-1}).$ One can choose $u_{s-1}u_{2n}\in E(G_{2n}).$ Then $C$ is a rainbow Hamilton cycle of $\mathcal{G}.$
\setcounter{case}{0}
\end{proof}
Now we are in a position to present the proof of Theorem \ref{thm1.6}.

\medskip
\noindent  \textbf{Proof of Theorem \ref{thm1.6}.}
Note that $Q^0_n$ is a proper subgraph of $B^1_n$ and $B^1_n$ is connected. Then we have $\rho(B^1_n)>\rho(Q^0_n).$ For $n=2,$ by Lemma \ref{lem2.5}, we have $\rho(G_i)\leq \sqrt{|E(G_i)|}.$ Then $|E(G_i)|\geq \rho^2(B^1_2)>\rho^2(Q^0_2)=2,$ which implies that $|E(G_i)|\geq 3.$ Hence $G_i\cong B^1_2$ or $G_i\cong K_{2,2}.$ Then $B^1_2$ is a subgraph of $G_i.$ If $|E(G_i)|=3$ for every $i\in[4],$ then $G_i\cong B^1_2.$ Suppose that $\mathcal{G}$ has no rainbow Hamilton cycle. By Lemma \ref{lem5.2}, $G_1=G_2=G_3=G_4\cong B^1_2.$ If $|E(G_j)|=4$ for some $j\in[4],$ then $G_j\cong K_{2,2}.$ Without loss of generality, assume that $j=1.$ Note that $B^1_2$ is a subgraph of $G_i$ for each $i\in [4].$ Let $G'_i\cong B^1_2$ be a subgraph of $G_i$ with $2\leq i\leq 4.$ Since $G_1\cong K_{2,2},$ there exist two subgraphs $G'_1$ and $G''_1$ of $G_1$ such that $G'_1\cong B^1_2,$ $G''_1\cong B^1_2$ and $G'_1\neq G''_1.$ Then $G'_1\neq G'_2$ or $G''_1\neq G'_2.$ By Lemma \ref{lem5.2}, $\{G'_1, G'_2, G'_3, G'_4\}$ or $\{G''_1, G'_2, G'_3, G'_4\}$ has a rainbow Hamilton cycle, which implies that so does $\mathcal{G}.$

For $n=3,$ by Lemma \ref{lem2.5}, we have $|E(G_i)|\geq 7$ for each $i\in[6].$ We claim that $B^1_3$ is a subgraph of $G_i.$ Denote by $H$ the graph obtained from $K_{3,3}$ by deleting two edges and $H\ncong B^1_3.$ If $|E(G_i)|= 7,$ then $G_i\cong B^1_3$ or $H.$ Note that $H$ is connected and $S(H)\cong B^1_3$. By Lemma \ref{lem2.2} and $H\ncong B^1_3$, we have $\rho(B^1_3)=\rho(S(H))>\rho(H).$ It follows from $\rho(G_i)\geq \rho(B^1_3)$ that $G_i\ncong H,$ which implies that $G_i\cong B^1_3.$ If $|E(G_i)|\geq 8,$ then $B^1_3$ is a proper subgraph of $G_i.$ Hence $B^1_3$ is a subgraph of $G_i.$ Using the same argument as the case of $n=2,$ we can also obtain that $\mathcal{G}$ admits a rainbow Hamilton cycle unless $G_1=G_2=\ldots = G_6\cong B^1_3.$

For $n\geq 4,$ we assume that $\mathcal{G}=\{G_1, G_2, \ldots,$ $ G_{2n}\}$ has no rainbow Hamilton cycle. For convenience, we denote $S_i=S(G_i)$ for $1\leq i\leq 2n.$ According to Lemma \ref{lem2.4}, the bi-shifted family $\{S_1,$ $S_2,$ $\ldots,$ $S_{2n}\}$ has no rainbow Hamilton cycle. By Lemma \ref{lem2.1} and $\rho(G_i)\geq \rho(B^1_n)$, we have
\begin{eqnarray}\label{eq::5}
    \rho(S_i)\geq \rho(G_i)\geq \rho(B^1_n).
\end{eqnarray}
\begin{claim}\label{claim::7}
For each $i\in [2n],$ $S_i$ is connected.
\end{claim}
\begin{proof}
Suppose to the contrary that $S_i$ is not connected. Then
\begin{eqnarray*}
    \rho(S_i)\leq \rho(K_{n,n-1})=\rho(Q^0_n)<\rho(B^1_n),
\end{eqnarray*}
which contradicts (\ref{eq::5}) .
\end{proof}
Next we show that $S_i\cong B^1_n$ for every $i\in[2n].$ Without loss of generality, suppose $X=[n]$ and $Y=[2n]\setminus[n].$ Define $e_j=\{j, 2n-j+2\},$ where $2\leq j\leq n.$
\begin{claim}\label{claim::8}
$\{e_3, e_4, \cdots, e_{n-1}\}\subseteq E(S_i)$ for each $i\in[2n].$
\end{claim}
\begin{proof}
If Claim \ref{claim::8} does not hold, then there exist some $3\leq p \leq n-1$ and some $q\in [2n]$ such that $e_p=\{p, 2n-p+2\} \notin E(S_q).$ Note that $S_q$ is bi-shifted. By Observation \ref{obs::1}, for every edge $\{i, j\}\in E(S_q),$  one can obtain that $1\leq i<p$ or $n+1\leq j <2n+2-p.$ Then $S_q$ is a subgraph of $B^{p-1}_n,$ which implies that $\rho(S_q)\leq \rho(B^{p-1}_n).$ By Lemma \ref{lem2.8},
\begin{eqnarray*}
    \rho(S_q)\leq \rho(B^{p-1}_n)<\rho(B^1_n),
\end{eqnarray*}
which contradicts (\ref{eq::5}).
\end{proof}
By Claim \ref{claim::7}, we have $\delta(S_i)\geq 1.$ If $d_{S_i}(2n)=1$ for each $i\in[2n],$ then $S_i$ is a subgraph of $B^1_n.$ It follows from (\ref{eq::5}) that $\rho(S_i)=\rho(B^1_n).$ Since $B^1_n$ is connected, we have $S_i\cong B^1_n$ for each $i\in[2n].$ Similarly, if $d_{S_i}(n)=1$ for each $i\in[2n],$ then we have $S_i\cong B^1_n.$ Next we assume that there exist $p, q\in [2n]$ such that $d_{s_p}(n)>1$ and $d_{s_q}(2n)>1.$ Recall that $S_p$ and $S_q$ are bi-shifted. By Observation \ref{obs::1}, we can obtain that $e_n=\{n, n+2\}\in E(S_p)$ and $e_2=\{2, 2n\}\in E(S_q).$ We divide the following proof into two cases.
\begin{case}\label{ca1}
    $p\neq q.$
\end{case}
Without loss of generality, we assume that $p=n$ and $q=2.$ Then $e_n\in E(S_n)$ and $e_2\in E(S_2).$  By $\delta(S_i)\geq 1$ and Observation \ref{obs::1}, we have $\{1, n+1\}\in E(S_1),$ $\{1, 2n\}\in E(S_{n+1})$ and $\{n, n+1\}\in E(S_{2n}).$ By Claim \ref{claim::8}, we choose $e_j\in E(S_j)$ for $3\leq j\leq n-1.$ Note that $e_j=\{j, 2n-j+2\}\in E(S_{n+j})$ for $3\leq j\leq n-1.$ Then we also obtain that $\{j, 2n-j+1\}\in E(S_{n+j})$ for $3\leq j\leq n-1.$ If $\{2, 2n-1\}\in E(S_{n+2}),$ then $\{S_1, S_2, \ldots, S_{2n}\}$ has a rainbow Hamilton cycle, a contradiction. If $\{2, 2n-1\}\notin E(S_{n+2}),$ then by Observation \ref{obs::1}, we have $d_{S_{n+2}}(2n-1)=d_{S_{n+2}}(2n)=1.$ Hence $S_{n+2}$ is a proper subgraph of $B^1_n.$ Since $B^1_n$ is connected, we have $\rho(S_{n+2})<\rho(B^1_n),$ which contradicts (\ref{eq::5}).
\begin{case}
    $p=q.$
\end{case}
Without loss of generality, we assume that $p=q=1.$ We first prove that $d_{S_i}(n)=1$ and $d_{S_i}(2n)=1$ for each $2\leq i\leq 2n.$ In fact, if there exists some $2\leq j\leq 2n$ such that $d_{S_j}(n)>1$ or $d_{S_j}(2n)>1,$ then we obtain that $d_{S_1}(n)>1$ and $d_{S_j}(2n)>1$, or $d_{S_1}(2n)>1$ and $d_{S_j}(n)>1.$ By case \ref{ca1}, $\{S_1, S_2, \ldots, S_{2n}\}$ has a rainbow Hamilton cycle, a contradiction. Hence $S_i$ is a proper subgraph of $B^1_n.$ Since $B^1_n$ is connected, we have $\rho(S_i)<\rho(B^1_n),$ which contradicts (\ref{eq::5}).

Therefore, $S_i\cong B^1_n$ for each $i\in [2n].$ Combining (\ref{eq::5}), we have $\rho(S_i)=\rho(G_i)=\rho(B^1_n).$ By Corollary \ref{cor5.1}, $G_i\cong B^1_n.$ If there exist $p,q\in [2n]$ such that $G_p\neq G_q,$ then by Lemma \ref{lem5.2}, $\mathcal{G}$ has a rainbow Hamilton cycle, a contradiction. Hence $G_1=G_2=\cdots=G_{2n}$ and $G_1\cong B^1_n.$
\hspace*{\fill}$\Box$

\vspace{5mm}
\noindent
{\bf Declaration of competing interest}
\vspace{3mm}

The authors declare that they have no known competing financial interests or personal relationships that could have appeared to influence the work reported in this paper.

\vspace{5mm}
\noindent
{\bf Data availability}
\vspace{3mm}

No data was used for the research described in this paper.



\end{document}